\documentclass[12pt]{amsart}
\usepackage{amsmath,amsfonts,textcomp,longtable,pstricks,xypic}
\usepackage[dvips]{graphicx}

\newtheorem{thm}{Theorem }[section]

\newtheorem{lemma}[thm]{Lemma }

\newtheorem{prop}[thm]{Proposition }

\newtheorem{corollary}[thm]{Corollary }

\theoremstyle{definition}
\newtheorem{deff}[thm]{Definition }

\newtheorem{rem}[thm]{Remark }

\newtheorem{ex}[thm]{Example }

\def\RR{{\mathbb R}}

\def\QQ{{\mathbb Q}}
\def\ZZ{{\mathbb Z}}
\def\NN{{\mathbb N}}

\def\FF{{\mathbb F}}
\def\fq{{\mathbb{F}_q}}
\def\kk{{k^{alg}}}

\def\SS{{\bar{S}}}

\def \bra#1\ket {\mathop{\vphantom{#1}\left<\smash{#1}\right>}\nolimits}

\DeclareMathOperator{\End}{End}

\DeclareMathOperator{\tr}{tr}

\DeclareMathOperator{\Gal}{Gal}
\DeclareMathOperator{\Spec}{Spec}

\DeclareMathOperator{\rk}{rk}

\DeclareMathOperator{\Ht}{Ht}
\DeclareMathOperator{\diag}{diag}

 \DeclareMathOperator{\Np}{Np}
\DeclareMathOperator{\Hp}{Yp}

\def\T{\Lambda}

\def\OO{\mathcal{O}}

\def\Fr{F}

\def\ch{\mathop{\mathrm{char}}\nolimits}
\def\plim{\mathop{{\lim\limits_{\longleftarrow}}}\nolimits}

\renewcommand \phi {\varphi}
\renewcommand \rho {\varrho}

\begin{document}
\author{Sergey Rybakov}
\thanks{The author is partially supported by AG Laboratory GU-HSE, RF government grant, ag. 11 11.G34.31.0023, 
and by RFBR grants no. 11-01-12072, 11-01-00395 and 10-01-93110-CNRSLa}
\thanks{E-mail address: rybakov@mccme.ru, rybakov.sergey@gmail.com}
\address{Poncelet laboratory (UMI 2615 of CNRS and Independent University of
Moscow)}
\address{Institute for information transmission problems of the Russian Academy of Sciences}
\address{Laboratory of Algebraic Geometry, NRU HSE, 7 Vavilova Str., Moscow, Russia, 117312}
\email{rybakov@mccme.ru, rybakov.sergey@gmail.com}%
\title[Finite group subschemes of abelian varieties]
{Finite group subschemes of abelian varieties\\ over finite
fields}
\date{}


\begin{abstract}
Let $A$ be an abelian variety over a finite field $k$. The
$k$-isogeny class of $A$ is uniquely determined by the Weil
polynomial $f_A$. We assume that $f_A$ is separable. For
a given prime number $\ell\neq\ch k$ we give a classification of
group schemes $B[\ell]$, where $B$ runs through the isogeny class,
in terms of certain Newton polygons associated to $f_A$. As an
application we classify zeta functions of Kummer surfaces over
$k$.
\end{abstract}
\maketitle

\section{Introduction.}
Throughout this paper $k$ is a finite field $\fq$ of characteristic $p$, and $\kk$ is an algebraic closure of $k$. Let $A$ be an abelian variety of dimension $g$ over $k$. Let
$A[m]$ be the group subscheme of $A$ annihilated by a natural number $m$. Fix a prime
number $\ell\neq p$. We say that $A[\ell]$ is the {\it
$\ell$-torsion of $A$}. In this paper we classify
$\ell$-torsion of abelian varieties in two cases: when the Weil polynomial is separable, and for abelian surfaces. This result is similar to the classification of groups of
$k$-points $A(k)$ (see~\cite{Ry4}). These two problems are closely
related, but the former one seems to be easier. 

Denote by $A_m=A[m](\kk)$ the kernel of multiplication by $m$ in
$A(\kk)$. Let $T_\ell(A) = \plim A_{\ell^r}$ be the $\ell$-th Tate module of $A$, and let $V_\ell(A)=T_\ell(A)\otimes_{\ZZ_\ell}\QQ_\ell$ be the
corresponding vector space over $\QQ_\ell$. Then $T_\ell(A)$ is a
free $\ZZ_\ell$-module of rank $2g$. The Frobenius endomorphism
$F$ of $A$ acts on the Tate module by a semisimple linear
transformation, which we also denote by $F: T_\ell(A)\to T_\ell(A)$. The
characteristic polynomial
$$
f_A(t) = \det(t-F)
$$
is called {\it the Weil polynomial of $A$}. It is a monic
polynomial of degree $2g$ with rational integer coefficients
independent of the choice of prime $\ell$. It is well known that
for isogenous varieties $A$ and $B$ we have $f_A(t)=f_B(t)$. Tate
proved that the isogeny class of an abelian variety is determined
by its characteristic polynomial, that is $f_A(t)=f_B(t)$ implies
that $A$ is isogenous to $B$~\cite{Ta66}.

This gives a nice description of isogeny classes of abelian
varieties over $k$ in terms of Weil polynomials. It seems natural
to consider classification problems concerning abelian varieties
inside a given isogeny class. Our goal is to describe the Frobenius action on
$\ell$-torsion of abelian varieties in a given isogeny class in terms of corresponding Weil polynomial.
Since $A[\ell](\kk)$ is an $\FF_\ell$-vector space, we have to describe possible matrices of the Frobenius action on such vector spaces.

In the second section we reduce the problem to a particular linear algebra question.
Here is a simplified version of the question. Let $N$ be a
nilpotent $d\times d$ matrix over $\FF_\ell$, and let $Q
\in\ZZ_\ell[t]$ be a polynomial of degree $d$ such that $Q\equiv t^d\bmod\ell$. Is it possible to find a
matrix $M$ over $\ZZ_\ell$ such that the characteristic polynomial
of $M$ is $Q$, and $M\equiv N\bmod\ell$? We will refer to this
question as {\it lifting of the nilpotent matrix $N$ to $\ZZ_\ell$
with respect to $Q$}.

The main results of the paper are proved in section $3$. First we associate to a nilpotent matrix $N$ a polygon of
special type. Let $m_1\geq\dots\geq m_r$ be the dimensions of the
Jordan cells of $N$. The numbers $m_1,\dots,m_r$ determine the
matrix up to conjugation. {\it The Young  polygon} $\Hp(N)$ of $N$
is the convex polygon with vertices $(\sum_{j=1}^{i}m_j,i)$ for
$0\leq i\leq r$. For a polynomial $Q\in\ZZ[t]$ we denote by
$\Np_\ell(Q)$ the Newton polygon of $Q$ with respect to $\ell$
(see Section $3$ for a precise definition). Assume that $Q$ is separable.
The main result of section $3$ can be reformulated as follows: one can lift $N$ to
$\ZZ_\ell$ with respect to $Q$ if and only if $\Np(Q)$ lies on or
above $\Hp(N)$ (see Theorems~\ref{main1} and~\ref{main2}). This
result allows one to classify $\ell$-torsion of abelian varieties
belonging to an isogeny class corresponding to the Weil polynomial
without multiple roots (Corollaries~\ref{main_av_tors1} and~\ref{main_av_tors2}).

In section $4$ we establish a relationship between Young polygons
for the Frobenius actions on an abelian variety and its dual. We
also treat the following question due to B. Poonen: is it true
that for an abelian surface $A$ the group of $k$-rational points
$A(k)$ is isomorphic to the the group of $k$-rational points
$\widehat A(k)$ on its dual? The answer is no, and we give a
counterexample.

In section $5$ we prove that (generalized) matrix factorizations correspond to Tate modules. 
This technique turns out to be useful when Weil polynomial is not separable.
In section $6$ we explicitly classify $\ell$-torsion of abelian
surfaces. In the final section we apply this result to the
classification of zeta functions of Kummer surfaces.

{\bf Acknowledgements.}
I am deeply grateful to Alexander Kuznetsov, who communicated his unfinished results on zeta functions of Kummer surfaces to me and provided many useful corrections on the early version of the paper. 
I thank Michael A.~Tsfasman for his attention to this work. 
I am grateful to Alexey Zykin and referees for suggesting many useful corrections and comments on the paper.

\section{Preliminaries}\label{end_rings}
\subsection*{Finite group subschemes of abelian varieties.}
A finite \'etale group scheme $G$ over $k$ is
uniquely determined by the Frobenius action on $G(\kk)$
(see~\cite{De}). If $\ell\cdot G=0$, then $G(\kk)$ is an
$\FF_\ell$-vector space and Frobenius action is $\FF_\ell$-linear.
By definition of the Tate module, we have $A[\ell](\kk)\cong
T_\ell(A)/\ell T_\ell(A)$. Thus the structure of a group scheme on
$A[\ell]$ depends only on the module structure on $T_\ell(A)$ over
$R=\ZZ_\ell[\Fr]\subset\End_k(A)$. Moreover, since the action of $F$ on
$V_\ell(A)$ is semisimple, $f_A$ determines the $R$-module $V_\ell(A)$ uniquely up to isomorphism. 
 
The following lemma shows what $R$-modules can arise as Tate
modules of varieties from a fixed isogeny class.

\begin{lemma}~\cite[IV.2.3]{Milne}
\label{lem_on_Tate_module} If $f:B\to A$ is an isogeny then,
$T_\ell(f):T_\ell(B)\to T_\ell(A)$ is an embedding of $R$-modules,
and if $T$ denotes its image then
\begin{equation}\label{t}
F(T)\subset T\quad\text{and}\quad T\otimes_{\ZZ_\ell}\QQ_\ell
\cong T_\ell(A)\otimes_{\ZZ_\ell}\QQ_\ell.
\end{equation}
Conversely, if $T\subset T_\ell(A)$ is a $\ZZ_\ell$-submodule such
that $(\ref{t})$ holds, then there exists an abelian variety $B$
defined over $k$ and an isogeny $f:B\to A$ such that $T_\ell(f)$
induces an isomorphism $T_\ell(B)\cong T$.\qed
\end{lemma}

\subsection*{Generalized Jordan form}
Let $K$ be a field, and let $\lambda$ be an algebraic number over $K$.
Put $L=K(\lambda)$. Take a vector space $L^r$ with a natural basis $v_1,\dots,v_r$. Let $M:L^r\to L^r$ be a linear transformation such that its matrix is a sum of Jordan cells with eigenvalue $\lambda$, i.e. $M=\lambda I_r+N$, where $I_r$ is the identity matrix and $N$ is a nilpotent matrix of dimension $r$. The set $\{\lambda^jv_i|1\leq i\leq r, 0\leq j\leq n-1\}$ is a basis of $L^r$ as a $K$-vector space. Denote by $J(\lambda, N)$ the matrix of $M$ in this basis. It is called {\it generalized Jordan cell}. We have the following generalization of the Jordan decomposition theorem. 

\begin{thm}\label{JF}
Let $M$ be a linear transformation of a $K$-vector space $V$ with the characteristic polynomial $P$. Suppose that any irreducible divisor of $P$ is separable. Let $\Delta$ be the set of roots of $P$, and $\T\subset\Delta$ be the image of a section of the natural map $\Delta\to\Delta/\Gal(K^{sep}/K)$, i.e. for any root $\delta\in\Delta$ there exists a unique $\lambda\in\T$ which is conjugate to $\delta$. Then there exists a basis of $V$ such that the matrix of $M$ is a direct sum of generalized Jordan cells $J(\lambda,N_\lambda)$ for $\lambda\in\T$. This data determines $M$ uniquely up to isomorphism over $K$.
\end{thm}
\proof
Let $P=\prod_{\lambda_\in\T} P_\lambda^{d_\lambda}$ be the decomposition of $P$ into a product of monic irreducible separable polynomials $P_\lambda\in K[t]$ such that $P_\lambda(\lambda)=0$ for any $\lambda\in\T$. Then by the Chinese remainder theorem $$\overline R=K[t]/P(t)K[t]\cong\prod_{\lambda\in\T} K[t]/P_\lambda(t)^{d_\lambda}K[t].$$
The vector space $V$ is an $\overline R$-module such that the image of $t$ in $\overline R$ acts on $V$ as $M$.
Put $\overline L_\lambda=K[t]/P_\lambda(t)K[t]$, and $\overline R_\lambda=K[t]/P_\lambda(t)^{d_\lambda}K[t]$. It follows that $V\cong\oplus V_\lambda$, where $V_\lambda=\overline R_\lambda V$ is an $\overline R_\lambda$-module. For any $\lambda\in\T$ the polynomial $P_\lambda$ is separable, thus $\Spec \overline L_\lambda$ is smooth over $\Spec K$. By~\cite[17.5.1]{ega4} (see also~\cite[II. exercise 8.6]{Ha}) we can find a section $\psi_\lambda$ of the natural morphism $\phi_\lambda:\overline R_\lambda\to \overline L_\lambda$, i.e.,  $\overline R_\lambda$ is an $\overline L_\lambda$-algebra, and $V_\lambda$ has a structure of an $\overline L_\lambda$-vector space. Denote by $t_\lambda\in \overline R_\lambda$ the image of $t$ under the natural projection $K[t]\to \overline R_\lambda$. Then $\lambda=\phi_\lambda(t_\lambda)$, and
$n_\lambda=t_\lambda-\psi_\lambda(\lambda)\in \overline R_\lambda$ is in the kernel of $\phi_\lambda$. Thus $n_\lambda^{d_\lambda}=0$, i.e. $n_\lambda$ acts on $V_\lambda$ as a nilpotent matrix $N_\lambda$. We see that $t_\lambda$ acts on $V_\lambda$ as a generalized Jordan cell $J(\lambda,N_\lambda)$. 

Finally, we have to prove that if $\oplus_{\lambda\in\T}J(\lambda,N_\lambda)$ is conjugate to $\oplus_{\lambda\in\T}J(\lambda,N'_\lambda)$ over $K$, then for all $\lambda\in\T$ the matrix $N_\lambda$ is conjugate to $N'_\lambda$ over $\overline L_\lambda$. Indeed, these matrices have the same dimension $d_\lambda$, and are conjugate by the Jordan decomposition theorem over $\overline L_\lambda$.
\qed

\begin{rem}\label{rem1}
We proved that $\overline R_\lambda\cong \overline L_\lambda[t]/(t-\lambda)^{d_\lambda}\overline L_\lambda[t]$. We use this isomorphism later.
\end{rem}

\subsection*{Reduction step 1.}
For a polynomial $P\in\ZZ_\ell[t]$ denote by $\Bar P\in\FF_\ell[t]$ its reduction modulo $\ell$, and by $P_1\in\ZZ_\ell[t]$ the unitary separable polynomial with the same set of roots as $P$. We call $P_1$ \emph{the minimal polynomial} of $P$.
The Frobenius action on $V_\ell(A)$ is semisimple, thus the minimal polynomial $f_1$ of $f$ is the minimal polynomial of the Frobenius action. It follows that $R\cong\ZZ_\ell[t]/f_1(t)\ZZ_\ell[t]$. 

The Galois group $\Gal(\overline{\FF}_\ell/\FF_\ell)$ acts on the set $\Delta$ of roots of $\Bar f_1$. Let $\T\subset\Delta$ be the image of a section of the natural map $\Delta\to\Delta/\Gal(\overline{\FF}_\ell/\FF_\ell)$.
By Theorem~\ref{JF} applied to the action of Frobenius on $T_\ell(A)/\ell T_\ell(A)$ the matrix of $F$ is conjugate to the sum of $J(\lambda, N_\lambda)$ for $\lambda\in\T$. We generalize this result to Tate modules.

By the Hensel lemma \cite[18.5.13]{ega4}, we can decompose $f_1$ into the product of monic polynomials $f_\lambda\in\ZZ_\ell[t]$ such that $\Bar f_\lambda$ is a power of an irreducible monic polynomial corresponding to $\lambda\in\T$. 
We have a natural homomorphism of rings
$$\phi:R\to\prod_{\lambda\in\T} \ZZ_\ell[t]/f_\lambda(t)\ZZ_\ell[t].$$
Since $P_\lambda$ is monic, $R_\lambda=\ZZ_\ell[t]/f_\lambda(t)\ZZ_\ell[t]$ is free and finitely generated $\ZZ_\ell$-module. 
By the Chinese remainder theorem $\phi$ is an isomorphism modulo $\ell$. On the other hand, $\phi$ is a homomorphism of finitely generated free $\ZZ_\ell$-modules. It follows that $\phi$ is an isomorphism.

The module $T_\ell(A)$ is an $R$-module such that the image of $t$ in $R$ acts as Frobenius. 
Put $T_\lambda=R_\lambda T_\ell(A)$, then $T_\ell(A)=\oplus T_\lambda$. By Theorem~\ref{JF}, the matrix of the action of $t$ on $T_\lambda/\ell T_\lambda$ is of the form $J(\lambda, N_\lambda)$ in some basis. We now sum up our observations. 

\begin{prop}\label{red1} There is an isomorphism of $R$-modules
$T_\ell(A)\cong\oplus_{\lambda\in\T} T_\lambda$ such that $F$ acts on $T_\lambda/\ell T_\lambda$ with matrix $J(\lambda, N_\lambda)$ in some basis.
\end{prop}

\subsection*{Reduction step 2.}
Let $L_\lambda$ be an unramified extension of $\QQ_\ell$ with residue field $\FF_\ell(\lambda)$. Denote by $S_\lambda$ the ring of integers of $L_\lambda$. 

\begin{prop}\label{red2}
There is an isomorphism $R_\lambda\cong S_\lambda[t]/g S_\lambda[t]$ for some $g\in S_\lambda[t]$ 
such that $g\equiv(t-\lambda)^d\bmod\ell S_\lambda$, where $d$ is the multiplicity of $\lambda$ in ${\Bar f}_\lambda$.
\end{prop}
\proof
Let $\Delta_\lambda$ be the set of roots of ${\Bar f}_\lambda$. Then 
$$ f_\lambda\equiv\prod_{\delta\in\Delta_\lambda}(t-\delta)^{d_\delta}\bmod\ell S_\lambda,$$ for some natural numbers $d_\delta$. By the Hensel lemma, $f_\lambda$ equals to a product of monic polynomials $g_\delta\in S_\lambda[t]$ such that $$g_\delta\equiv (t-\delta)^{d_\delta}\bmod\ell S_\lambda,$$ where $\delta\in\Delta_\lambda$. Define the homomorphism of rings $$R_\lambda\to R_\lambda\otimes_{\ZZ_\ell} S_\lambda$$ by $r\mapsto r\otimes 1$. By the Chinese remainder theorem,
$$R_\lambda\otimes_{\ZZ_\ell} S_\lambda\cong \prod_{\delta\in\Delta_\lambda}Z_\delta$$ where $Z_\delta\cong S_\lambda[t]/g_\delta S_\lambda[t]$. Take the projection $R_\lambda\otimes_{\ZZ_\ell} S_\lambda\to Z_\lambda$. We get a homomorphism $\phi:R_\lambda\to Z_\lambda$. It is a homomorphism of free $\ZZ_\ell$-modules and an isomorphism modulo $\ell$ by Remark~\ref{rem1}. We conclude that $\phi$ is an isomorphism. Put $g=g_\lambda$. 
\qed

Choose an element $\alpha\in S_\lambda$ such that $\Bar\alpha=\lambda$.
The polynomial $Q_\lambda(t)=g(t-\alpha)$ is the minimal polynomial of $F-\alpha$ acting on $T_\lambda$. Clearly, $Q_\lambda\equiv t^d\bmod\ell$, where $d=\deg Q_\lambda$. We have reduced our task to the following linear algebra problem.


\subsection*{The problem.}
Let $L$ be an unramified extension of $\QQ_\ell$, and let $S$ be its ring of integers. Suppose we are given a polynomial $Q\in S[t]$ such that $Q\equiv t^d\bmod\ell$, where $d=\deg Q$. Let $V$ be an $L$-vector space of dimension $d$, and let $E$ be a semisimple linear transformation on $V$ with characteristic polynomial $Q$. Denote by $Q_1$ the minimal polynomial of $E$. Put $R=S[t]/Q_1(t)S[t]$. We give a structure of an $R$-module on $V$ such that $t$ acts as $E$. Describe all isomorphism classes of finite $R$-modules of the form $T/\ell T$,
where $T$ is an arbitrary $R$-invariant $S$-lattice in $V$.

If we choose a basis of $T$, the problem can be reformulated as
follows. Let $N$ be the matrix of the action of $E$ on
$T/\ell T$ in some basis over the finite field $S/\ell S$. It is a
nilpotent matrix over $S/\ell S$, since $Q\equiv t^d\bmod\ell$. Is
it possible to find a matrix $M$ over $S$ such that
$Q(t)=\det(t-M)$, and $M\equiv N\bmod\ell$? We will refer to this
question as {\it lifting of nilpotent matrix $N$ to $S$ with
respect to $Q$}.

\section{$\ell$-torsion of abelian varieties}
\label{ltors} Let $S$ be the ring of integers in an unramified extension $L$ of $\QQ_\ell$.
Assume we are given a finitely generated free $S$-module $T$ endowed with an $S$-linear injective
endomorphism $E$ which induces on $T/\ell T$ a nilpotent endomorphism $N$. Let $Q(t)=\det(t-E)$. 
In this section we give a partial answer to the question: when is it
possible to lift $N$ to $S$ with respect to $Q$? Using this result
we get a classification of group schemes of the form $A[\ell]$ for
$A$ from a fixed isogeny class such that $f_A$ is separable.

We associate to $N$ a polygon of special type. For a sequence of natural numbers $m_1\geq\dots\geq m_r> 0$ we define {\it the Young polygon} $\Hp(m_1,\dots,m_r)$ as the convex polygon
with vertices $(\sum_{j=1}^{i}m_j,i)$ for $0\leq i\leq r$. The {\it dimension of} $Y=\Hp(m_1,\dots,m_r)$ is $\dim Y=\sum_{j=1}^{r}m_j$. The {\it height of} $Y$ is $\Ht(Y)=r$.

There is a basis of $T/\ell T$ such that the matrix of $N$ is a sum of Jordan cells of dimensions $m_1,\dots, m_r$.
Clearly, numbers $m_1,\dots, m_r$ determine $N$ uniquely up to conjugation.
We associate to $N$ the Young polygon $\Hp(N)$ given by the sequence $m_1,\dots, m_r$.
We also denote this Young polygon by $\Hp(E|T)$. 

The Young polygon has $(0,0)$ and $(d,r)$ as its endpoints, and
its slopes are $1/m_1,\dots,1/m_r$. For example, the following
picture shows Young polygons for the zero matrix (Pic. 1) and the Jordan cell of dimension two (Pic. 2).

\begin{pspicture}(200pt,110pt)
\psline[linewidth=0.3pt](80pt,30pt)(10pt,30pt)(10pt,100pt)
\psline[linewidth=1.2pt](10pt,30pt)(70pt,90pt)

\pscircle*(10pt,30pt){1.5pt} \pscircle*(40pt,30pt){1.5pt}
\pscircle*(40pt,60pt){1.5pt} \pscircle*(10pt,90pt){1.5pt}
\pscircle*(70pt,30pt){1.5pt} \pscircle*(70pt,90pt){1.5pt}
\pscircle*(10pt,60pt){1.5pt}

\rput(5pt,23pt){0} \rput(37pt,23pt){1} \rput(67pt,23pt){2}
 \rput(3pt,87pt){2} \rput(3pt,57pt){1}

\rput(40pt,5pt){Pic. 1}

\psline[linewidth=0.3pt](180pt,30pt)(110pt,30pt)(110pt,100pt)
\psline[linewidth=1.2pt](110pt,30pt)(170pt,60pt)

\pscircle*(110pt,30pt){1.5pt} \pscircle*(140pt,30pt){1.5pt}
\pscircle*(110pt,90pt){1.5pt} \pscircle*(170pt,30pt){1.5pt}
\pscircle*(110pt,60pt){1.5pt} \pscircle*(170pt,60pt){1.5pt}

\rput(105pt,23pt){0} \rput(137pt,23pt){1} \rput(167pt,23pt){2}
\rput(103pt,87pt){2} \rput(103pt,57pt){1}

\rput(140pt,5pt){Pic. 2}
\end{pspicture}

\bigskip

Denote by $\nu$ the normalized valuation on $L$, i.e. $\nu(\ell)=1$. Suppose that $Q(t)=\sum_i Q_i t^{d-i}$. Take the lower convex hull of the points $(i,\nu(Q_i))$ for $0\leq i\leq \deg Q$ in $\RR^2$. The boundary
of this region is called {\it the Newton polygon $\Np(Q)$ of $Q$}. Its vertices have integer coefficients, and $(0,0)$ and $(d,\nu(Q_d))$ are its endpoints. The {\it slopes of $Q$} are the slopes of this polygon. Note that each slope has a multiplicity. 

\begin{thm}{\label{main1}}
The Newton polygon $\Np(Q)$ lies on or above Young polygon $\Hp(E|T)$.
\end{thm}
\begin{proof}
Let $Q_1\in S[t]$ be the minimal polynomial of $E$, and let $R=S[t]/Q_1(t)S[t]$.
Let $x\in R$ be the image of $t$ under the natural projection. The module $T$ is naturally an $R$-module such that $x$ acts as $E$.

Suppose $\Hp(E|T)=\Hp(m_1,\dots,m_r)$. Take generators $v'_1,\dots, v'_r$ of $T/\ell T$ over $R$ such that
$$v'_1,xv'_1\dots,x^{m_1-1}v'_1,\dots, v'_r,\dots, x^{m_r-1}v'_r$$ is a Jordan basis for $x$. 
By the Nakayama lemma there exist generators $v_1,\dots, v_r$ of $T$ over $R$ which lift $v'_1,\dots, v'_r$.
Let $H$ be a matrix of $x$ in the basis 
$$v_1,xv_1\dots,x^{m_1-1}v_1,\dots, v_r,\dots, x^{m_r-1}v_r,$$
and let $H_{i_1,\dots,i_m}$ be the determinant of
the submatrix of $H$ cut by the columns and rows with the numbers
$i_1,\dots,i_m$. The characteristic polynomial of $x$ acting on
$T$ is $$Q(t)=\sum_{m=0}^d Q_mt^{d-m},$$ and
$$Q_m=(-1)^m\sum_{i_1<\dots<i_m}H_{i_1,\dots,i_m}.$$ It follows
that
$$\nu(Q_m)\geq\min_{i_1<\dots<i_m}\nu(H_{i_1,\dots,i_m}).$$

Let $m=m_1+\dots+m_{s-1}+a$, where $0< a\leq m_s$. We have to show
that if $H_{i_1,\dots,i_m}\neq 0$, then
$\nu(H_{i_1,\dots,i_m})\geq s$. Note that if $i\neq m_1+\dots+m_j$
for all $j$, then the $i$-th column of $H$ has $1$ only in the
position number $i+1$, and its other entries are zero. Thus if
$i\in \{i_1,\dots,i_m\}$, and $H_{i_1,\dots,i_m}\neq 0$, then
$i+1\in \{i_1,\dots,i_m\}$. If $i= m_1+\dots+m_j$ for some $j$,
then $\ell$ divides the $i$-th column. We see that if
$H_{i_1,\dots,i_m}\neq 0$, then the set $\{i_1,\dots,i_m\}$ is a
union of {\it blocks}. Each block is an interval
$$\{i,i+1,\dots,m_1+\dots+m_j\}$$ of length not greater than $m_j$.
If $m>m_1+\dots+m_{s-1}$, then the set $\{i_1,\dots,i_m\}$
contains no less than $s$ blocks, and $\nu(H_{i_1,\dots,i_m})\geq
s$. Thus if $(m,\nu(Q_m))$ is a vertex of $\Np(Q)$ then
$\nu(Q_m)\geq s$, and $\Np(Q)$ lies on or above $\Hp(E|T)$.
\end{proof}

\begin{thm}{\label{main2}}
Let $R=S[t]/Q(t)S[t]$, and let $V=R\otimes_{\ZZ_\ell} \QQ_\ell$.
Let $Y$ be a Young polygon such that $\Np(Q)$ lies on or above
$Y$. Then there exists an $R$-lattice $T$ in $V$ such that
$\Hp(x|T)=Y$.
\end{thm}

\begin{proof}
Recall that $Q=\det(t-x|V)$ is the minimal polynomial of the
action of $x$ on $V$. Let $Y=Y(m_1,\dots,m_r)$. 
First, we find a lattice $T$ in $V$ over $S$ such that
$R\subset T\subset V$. After that we prove that $T$ is an $R$-module.

Let $m=m_1+\dots+m_s$, and let
$Q(t)=\sum_{i=0}^d Q_it^{d-i}$. For $1\leq s\leq r$ we put
$$v_{s+1}=\frac{x^m+\sum_{j=1}^m Q_jx^{m-j}}{\ell^s}.$$
In addition, let $v_1=1$, and let $v_{r+1}=0$. Note that
$$v_1,xv_1\dots,x^{m_1-1}v_1,\dots, v_r,\dots, x^{m_r-1}v_r$$ have
different degrees viewed as polynomials in $x$, and hence generate a lattice 
$T$ over $S$.

Now we prove that $T$ is an $R$-module. The point $(m-m_s,s-1)$ is
a vertex of $Y$. By assumption, $\Np(Q)$ lies on or above
$Y$, thus $(m-m_s,s-1)$ is not higher than $\Np(Q)$. It
follows that $\ell^s$ divides $Q_j$ for all $j>m-m_s$. Thus $$
u_s=\frac{\sum_{j=m-m_s+1}^{m} Q_jx^{m-j}}{\ell^s}\in S\cdot
1\subset T.$$ Moreover, $$x^{m_s}v_s=\ell (v_{s+1}-u_s)\in \ell
T.$$ This proves that $xT\subset T$, and that $\Hp(x|T)=Y$.
\end{proof}

\begin{ex}Let $Q(t)=t^2-\ell t -\ell$. Its Newton polygon is drawn on
Picture~2. Then we can lift the nonzero nilpotent Jordan cell (its
Young polygon is equal to $\Np(Q)$). For example, take
$$M=\begin{pmatrix}
0&\ell  \\
1&\ell  \\
\end{pmatrix}.$$

Clearly, $Q(t)=\det(t-M)$. We can not lift the zero matrix,
because its Young polygon (see Pic. 1) is higher than $\Np(Q)$.
\end{ex}

Note that if $Q$ is not separable, then the action of $x$ on $V$ is not semisimple.  By Theorems~\ref{main1} and~\ref{main2} we get

\begin{corollary}
Suppose that $Q$ is separable. One can lift $N$ to $S$ with respect to $Q$ if and only if $\Np(Q)$ lies on or above $\Hp(N)$. 
\end{corollary}

Suppose that $Q$ is not separable, and $N$ is a nilpotent matrix such that $\Np(Q)$ lies on or above $\Hp(N)$.
Then in general it is not possible to lift $N$ to $S$ with respect to $Q$. We discuss a partial solution of this problem in section~\ref{mf}. Now we prove the following simple result.

\begin{prop}\label{np_deg2}
Suppose $Q=P^r$, where $\deg P=2$, and $P$ is separable. Let $R=S[t]/P(t)S[t]$, and let $V=(R\otimes_{\ZZ_\ell} \QQ_\ell)^r$. There exists an $S$-lattice $T$ in $V$ such
that $x$ acts on $T/\ell T$ with Young polygon $Y$ if and
only if $\Np(P^r)$ lies on or above $Y$, and all slopes of
$Y$ are equal to $1/2$ or $1$.\qed
\end{prop}
\begin{proof} Let $N$ be a nilpotent matrix with Young polygon $Y$.
Suppose such a lattice $T$ exists. Since $\deg P=2$, any Jordan
cell of $N$ has dimension at most $2$, thus all slopes of
$\Hp(N)$ are equal to $1/2$ or $1$. Conversely, there is a decomposition $N=\oplus N_i$
such that $\dim\Hp(N_i)=2$. By Theorem~\ref{main2} for any $i$ there exists an $S$-lattice $T_i$
such that $x$ acts on $T_i/\ell T_i$ with the matrix $N_i$. Put
$T=\oplus T_i$.
\end{proof}

We call a polynomial $f\in\ZZ_\ell[t]$ {\it distinguished} if $\Bar f$ is a power of an irreducible polynomial. We now use notation of section~\ref{end_rings}. Let $f_1\in\ZZ_\ell[t]$ be the minimal polynomial of $f$. Choose a root $\lambda$ of $\Bar f$ and its lifting $\alpha_\lambda\in S_\lambda$. By Proposition~\ref{red2}, $$R_\lambda=\ZZ_\ell[t]/f_1\ZZ_\ell[t]\cong S_\lambda[t]/gS_\lambda[t],$$ and $g\equiv(t-\lambda)^d\bmod\ell S_\lambda$. Put $Q_1(t)=g(t-\alpha_\lambda)$.  Note that $Q_1(t)$ divides $f(t-\alpha_\lambda)$ over $L_\lambda$.
Take a unitary polynomial $Q=Q_\lambda\in S_\lambda[t]$ of maximal degree such that $Q(t)$ divides $f(t-\alpha_\lambda)$ over $L_\lambda$, and the minimal polynomial of $Q$ is $Q_1$. We could define $Q$ in other way. Let $V$ be a $\QQ_\ell$--vector space endowed with semisimple linear transformation $F$ such that $f(t)=\det(t-F)$. Then $V$ is an $L_\lambda$--vector space, and $Q(t)$ is the characteristic polynomial of the $L_\lambda$--linear transformation $F-\alpha_\lambda$.

Recall that an \'etale group scheme is uniquely determined by the linear Frobenius action on the group of $\kk$-points.
Let $Y$ be a Young polygon of dimension $\deg Q$ such that $\Np(Q)$ lies on or above $Y$, and let $N$ be a nilpotent matrix such that $Y=\Hp(N)$. {\it A distinguished group scheme} is a finite \'etale group scheme  $A(f,Y)$ over $k$ such that $\dim_{\FF_\ell} A(f,Y)(\kk)=\deg f$, and $F$ acts on $A(f,Y)(\kk)$ with the matrix $J(\lambda,N)$ in some basis.
Note that for a given $f$ the polynomial $Q$ is uniquely determined modulo $\ell$. Thus $A(f,Y)$ is uniquely determined by $f$ and $Y$ up to an isomorphism.


\begin{corollary}\label{main_av_tors1}
Let $A$ be an abelian variety over $k$. Then $A[\ell]$ is isomorphic to a sum of distinguished group schemes.
\end{corollary}
\begin{proof}
Let $f_A=\prod_{\lambda\in\T} f_\lambda$ be a product of pairwise coprime distinguished polynomials. By Proposition~\ref{red1}, $T_\ell(A)\cong\oplus_{\lambda\in\T} T_\lambda$. By Proposition~\ref{red2}, $R_\lambda\cong S_\lambda[t]/gS_\lambda[t]$. This gives a structure of an $S_\lambda$-module on $T_\lambda$. As before, let $\alpha_\lambda\in S_\lambda$ be a lifting of $\lambda$, and let
$Q_\lambda$ be the characteristic polynomial of the action of $\Fr-\alpha_\lambda$ on $T_\lambda$. By Theorem~\ref{main1}, $\Fr$ acts on $T_\lambda/\ell T_\lambda$ with the matrix $J(\lambda,N_\lambda)$, where $N_\lambda$ is a nilpotent matrix such that $\Np(Q_\lambda)$ lies on or above $\Hp(N_\lambda)$. Thus $A[\ell]\cong\oplus_{\lambda\in\T} A(f_\lambda,\Hp(N_\lambda))$.
\end{proof}

\begin{corollary}\label{main_av_tors2}
Let $A$ be an abelian variety over $k$. Suppose $f_A$ is separable, and $f_A=\prod_{\lambda\in\T} f_\lambda$ is a product of coprime distinguished polynomials. Then for any family of
distinguished group schemes $A(f_\lambda,Y_\lambda)$ there exists an abelian
variety $B$ isogenous to $A$ such that $B[\ell]\cong\oplus_{\lambda\in\T}
A(f_\lambda,Y_\lambda)$.
\end{corollary}
\begin{proof}
By Proposition~\ref{red1}, $V_\ell(A)\cong\oplus_{\lambda\in\T} V_\lambda$, where
$V_\lambda=R_\lambda V_\ell(A)$ is an $R_\lambda$-module. Let $\alpha_\lambda$ be a lift of $\lambda$.
By Theorem~\ref{main2}, there exists an $S_\lambda$-lattice $T_\lambda\subset
V_\lambda$ such that $F-\alpha_\lambda$ acts on $T_\lambda/\ell T_\lambda$ with Young polygon $Y_\lambda$. 
Put $T=\oplus T_\lambda$. By Lemma~\ref{lem_on_Tate_module}, there exists a variety $B$ such that $T\cong T_\ell(B)$.
\end{proof}

\section{Young  polygons and duality.}
By $\widehat A$ we denote the dual variety of an abelian variety
$A$. Suppose $f$ is a distinguished polynomial
such that $f$ divides $f_A$, and polynomials $f$ and $f_A/f$ have
no common roots modulo $\ell$. By Proposition~\ref{red1}, there exists a
direct summand $T$ of $T_\ell(A)$ such that $F$ acts on $T$ with
characteristic polynomial $f$. The polynomial $\Hat f(t)=(\frac tq)^{\deg
f}f(\frac qt)$ divides $f_A$. Clearly, $\Hat f(t)$ is distinguished. Denote the corresponding direct summand of
$T_\ell(\widehat A)$ by $\widehat T$. 

Let $\lambda$ be a root of $\Bar f$, and let $S$ be the ring of integers in an unramified extension of $\QQ_\ell$ with residue field $\FF_\ell(\lambda)$. By Proposition~\ref{red2}, $T$ is an $S$-module. Clearly, $q/\lambda$ is a root of $\Hat f(t)$, and $\widehat T$ is an $S$-module too. 

\begin{prop}
Let $\alpha\in S$ be a lift of $\lambda$. Then $\Hp(F-\alpha|T)=\Hp(F-q/\alpha|\widehat T)$.
\end{prop}
\begin{proof}
The Weil pairing $e:T_\ell(A)\times T_\ell(\widehat A)\to
\ZZ_\ell$ is non-degenerate, and $e(Fx,Fy)=q e(x,y)$, where $x\in T_\ell(A)$
and $y\in T_\ell(\widehat A)$~\cite{Mum}. This yields that its restriction to
$T\times \widehat T$ is non-degenerate. 
By an integral version of Deligne trick~\cite[Lemma 3.1]{BGK},
there exists an $S$-linear pairing $e_S:T\times \widehat T\to S$
such that $e_S(Fx,Fy)=q e_S(x,y)$, and $e=Tr_{L/\QQ_\ell}\circ
e_S$, where $L$ is the fraction field of $S$. We have
$$e_S(Fx,y)=e_S(Fx,F(F^{-1}y))=e_S(x,(qF^{-1})y).$$ Let
$M=J(\lambda,N)$ be the matrix of the action of $F$ on $T/\ell T$ in some basis
over $S/\ell S$, and let $\widehat M$ be the matrix of the action
of $F$ on $\widehat T/\ell \widehat T$ in the dual basis. It
follows that $\widehat M^{t}=qM^{-1}$, where ${\cdot}^t$ means
transpose. One easily proves that for any cell of $M$ corresponding to the
Jordan cell of dimension $d$ there exists a cell of $\widehat M$
corresponding to the same Jordan cell. The proposition follows.
\end{proof}

We now give an example of an abelian surface $A$ such that the
group of points $A(k)$ is not isomorphic to the group of points on
the dual surface $\widehat A(k)$. Recall that $A(k)$ is a kernel
of $1-F: A\to A$, and the $\ell$-component $A(k)_\ell=\ker (1-F):
T_\ell(A)\to T_\ell(A)$.

\begin{ex} Let $q=7$, and let $\ell=5$. Suppose $f_a(t)=t^2+2t+7$ and
$f_b(t)=t^2-3t+7$ are Weil polynomials of two elliptic curves. The
polynomial $f=f_af_b$ is the Weil polynomial of an abelian
surface. Note that $f_a(t)\equiv f_b(t)\equiv (t-1)(t-q)\mod 5$.
Thus we have a decomposition $f=f_1f_2$ over $\ZZ_5$, where
$f_1\equiv(t-1)^2\mod 5$, and $f_2\equiv(t-q)^2\mod 5$.  For any
abelian surface $B$ with Weil polynomial $f$ we have a
decomposition $T_5(B)\cong T_1\oplus T_2$, where $F$ acts on $T_i$
with characteristic polynomial $f_i$ for $i=1,2$. By
Theorem~\ref{main1}, $F-1$ acts on $T_1/5T_1$ trivially and by
Theorem~\ref{main2}, there exists a lattice $T$ in $T_2\otimes\QQ$
such that $F-q$ acts on $T/5T$ non-trivially. In the first case the
Young  polygon of $F-1$ is $\Hp(2)$, and in the second case the
Young polygon of $F-q$ is $\Hp(1,1)$. By
Lemma~\ref{lem_on_Tate_module}, there exists an abelian surface $A$
such that $T_5(A)\cong T_1\oplus T_2$. By the previous proposition,
$A(\FF_7)_5\cong \ZZ/5\ZZ\oplus\ZZ/5\ZZ$, and $\widehat
A(\FF_7)_5\cong \ZZ/25\ZZ$.
\end{ex}

\section{Matrix factorizations.}\label{mf}
In this section we turn our attention to the case when the Weil polynomial is not separable. 
First we establish a connection between matrix factorizations and Tate modules.
Given a Tate module $T$ with the Weil polynomial $f$ and the minimal polynomial $f_1$
one can produce the unique (up to isomorphism) Tate module $T'$ with the Weil polynomial $g=f_1^r/f$ for some $r$. If we are lucky, $g$ do not have multiple roots. 
In this case, one can get some information on $T$ from $T'$ using theorem~\ref{main1}. Moreover, one can reverse the construction and produce 
a Tate module $T$ with a given Young polygon starting from $T'$ constructed using theorem~\ref{main2}.

Let $S$ be the ring of integers in a finite unramified extension
$L$ of $\QQ_\ell$. Fix a pair of polynomials $f, f_1\in S[t]$ and
a positive integer $r$.
Let $R=S[t]/f_1S[t]$, and let $\SS=S/\ell S$. Denote by $x\in R$
the image of $t$ under the natural projection from $S[t]$.
We assume that $f_1\equiv t^{d_1}\mod\ell$, and $\deg f_1=d_1$.

\begin{deff}
A {matrix factorization} (with respect to $f, f_1$ and $r$) is a pair $(X,Y)$ of $r\times r$ matrices
with coefficients in $S[t]$ such that $YX=f_1\cdot I_r$ and $\det
X=f$.
\end{deff}

Suppose we are given a matrix factorization $(X,Y)$.
The matrix $X$ defines a map of free $S[t]$ modules:
$$S[t]^r \xrightarrow{X} S[t]^r.$$ 
Its cokernel $T$ is annihilated by $f_1$. It is equivalent to say that $T$ is an $R$-module.
We see that the matrix factorization $(X,Y)$ corresponds to a finitely generated
$R$-module $T$ given by the presentation:
\begin{equation}\label{eq_mf}
S[t]^r \xrightarrow{X} S[t]^r \to T \to 0.
\end{equation}

\begin{prop}\label{prop_mf1}
The module $T$ is free of finite rank $d$ over $S$, and characteristic
polynomial of the action of $x$ on $T$ is equal to $f$.
\end{prop}
\begin{proof}
Since $f_1\equiv t^{d_1}\mod\ell$, and $\deg f_1=d_1$, the ring $R$
is generated as $S$-module by the elements $1,x,\dots,x^{d_1-1}$.
By definition, $T$ is a finitely generated $R$-module, thus it is
finitely generated over $S$.

Take the tensor product of the presentation ${\rm(\ref{eq_mf})}$
with $\SS[t]$: $$\SS[t]^r \xrightarrow{\bar X} \SS[t]^r \to
T\otimes_S\SS \to 0.$$ The ring $\SS[t]$ is a principal ideal
domain, thus there exist matrices $M_1$ and $M_2$ over $\SS[t]$
such that $\det M_1=\det M_2=1$ and $M_1\bar XM_2$ is the diagonal
matrix with determinant $t^d$. It follows that $M_1\bar XM_2$ is the
diagonal matrix $\diag(t^{m_1},\dots,t^{m_r})$ for some
$m_1\dots,m_r\in\NN$ such that $\sum m_i=d$. We get
$$T\otimes_S\SS\cong\oplus_{i=1}^r
\SS[t]/t^{m_i}\SS[t].$$ By the Nakayama lemma, $T$ is generated by $d$
elements over $S$.

Now take the presentation of $T\otimes_SL$:
$$L[t]^r \xrightarrow{X} L[t]^r \to
T\otimes_SL \to 0.$$ As before, there exist matrices $M_3$ and
$M_4$ over $L[t]$ such that $\det M_3=\det M_4=1$, and
$M_3XM_4=\diag(g_1,\dots,g_r)$. Clearly,
$$T\otimes_SL\cong\oplus_{i=1}^r L[t]/g_iL[t],$$
and $\rk T=d$. This proves that $T$ is free over $S$. To conclude
the proof we note that the characteristic polynomial of the
action of $x$ on $L[t]/g_iL[t]$ is equal to $g_i$.
\end{proof}

The following proposition shows that modules over $R$ give rise to
matrix factorizations.

\begin{prop}\label{prop_mf2}
Let $T$ be an $R$-module which is free of finite rank over $S$.
Suppose that $T$ is generated over $R$ by $r$ elements, and
that $\Hp(x|T)=\Hp(m_1,\dots,m_r)$. Then there exists a matrix
factorization $(X,Y)$ such that $T$ has presentation
${\rm(\ref{eq_mf})}$, and $$X\equiv
\diag(t^{m_1},\dots,t^{m_r})\mod\ell.$$
\end{prop}
\begin{proof}
Let $v_1,\dots, v_r$ be generators of $T$ over $R$. Then
$x^{m_i}v_i=\sum_j a_{ji}(x)v_j$, where $a_{ji}\in S[t]$ and $\deg
a_{ji}< m_j$. Let $X$ be the matrix with the entries
$t^{m_i}\delta_{ji}-a_{ji}$. Define an $R$-module $T'$ by the
presentation: $$S[t]^r \xrightarrow{X} S[t]^r \to T' \to 0.$$ Put
$m=\sum_i m_i$. Then $\det X\equiv t^m\mod\ell$, and from the
inequalities $\deg a_{ji}< m_j$ it follows that $\det X$ is a
polynomial of degree $m$. By Proposition~\ref{prop_mf1}, $T'$ is a
free $S$ module of rank $m$. By definition of $T'$, we have a
surjective map of $S$-modules $T'\to T$. Since they have the same
rank as $S$-modules, this map is an isomorphism, and, by
Proposition~\ref{prop_mf1}, $\det X=f$.

Multiplying presentation ${\rm(\ref{eq_mf})}$ by $f_1$ we get the
commutative diagram
$$\xymatrix{
S[t]^r \ar[r]^{X}\ar[d]_{f_1} & S[t]^r \ar[d]^{f_1} \ar[r] &T \ar[r] \ar[d]_{0}&0 \\
S[t]^r \ar[r]^{X} & S[t]^r \ar[r] &T \ar[r] &0 } $$

Since $S[t]^r$ is free, there exists a matrix $Y$ such that the
diagram
$$\xymatrix{
S[t]^r \ar[r]^{X} & \ar[dl]_{Y} S[t]^r \ar[d]^{f_1} \ar[r] &T \ar[r] \ar[d]_{0}&0 \\
S[t]^r \ar[r]^{X} & S[t]^r \ar[r] &T \ar[r] &0 } $$ commutes. It
follows
that $YX=f_1 I_r$. 
Thus, the pair $(X,Y)$ is a matrix factorization.
\end{proof}

\begin{ex}
Let $\deg f_1=3$, and let $f=f_1^2$. Suppose $f_1$ is separable. 
When there exists an $R$-module $T$ such that $x$ acts with $r=3$ Jordan cells of dimension $2$? 
By Proposition~\ref{prop_mf2}, such a module exists iff there exists a matrix factorization $(X,Y)$ such that
$$X\equiv \diag(t^2,t^2,t^2)\mod\ell.$$ The matrix
factorization $(Y,X)$ gives a module $T'$ over $R$ which is
generated by $3$ elements and the characteristic polynomial of $x$
is equal to $\det Y=f_1^3/f=f_1$. Moreover, $Y\equiv
\diag(t,t,t)\mod\ell.$ It follows that $\Hp(x|T')=\Hp(1,1,1)$. By
Theorems~\ref{main1} and~\ref{main2}, such a module $T'$ exists iff
$\Np(f_1)$ lies on or above $\Hp(x|T')$. Thus $T$ exists iff $\Np(f_1)$ lies on or above $\Hp(1,1,1)$.
\end{ex}

\section{$\ell$-torsion of abelian surfaces.}\label{sec_surf_tors}
In this section we classify isomorphism classes of $\ell$-torsion subschemes of abelian surfaces.
We use the following notation. 
Let $P=\prod_{\lambda\in\T}P_\lambda$ be the decomposition of a polynomial $P(t)\in\ZZ_\ell[t]$ into a product of distinguished polynomials. Then $A(P,0)$ is the group scheme $\oplus_{\lambda\in\T} A(P_\lambda,Y_{\deg P_\lambda})$, where $Y_n$ is the Young polygon of the zero matrix of dimension $n$. 

\begin{thm}\label{surf_tors}
Let $A$ be an abelian surface over $k$ with the Weil polynomial
$f_A(t)=t^4+a_1t^3+a_2t^2+qa_1t+q^2$. Suppose first that $f_A$ is separable, then we have the following five cases:
\begin{description}
\item[(1)] if $f_A(t)$ is separable modulo $\ell$, then $A[\ell]\cong A(f_A,0)$;
\item[(2)] if $f_A(t)\equiv f_1f_2 \bmod\ell$, where
$f_1\equiv(t-\alpha)^2\bmod\ell$, and $f_2$ is separable modulo
$\ell$, then $A[\ell]\cong A(f_2,0)\oplus A(f_1,Y)$, where $\dim Y=2$;
\item[(3)] if $f_A(t)\equiv f_1f_2 \bmod\ell$, where
$f_i\equiv (t-\alpha_i)^2\bmod\ell$, for $i=1,2$ and
$\alpha_1\not\equiv\alpha_2\mod\ell$, then $A[\ell]\cong
A(f_1,Y_1)\oplus A(f_2,Y_2)$, where $\dim Y_1=\dim Y_2=2$;
\item[(4)] if $f_A(t)\equiv h(t)^2\bmod\ell$, where $h$ is irreducible modulo
$\ell$, then $A[\ell]\cong A(f_A,Y)$, where $\dim Y=2$. If $\ell\neq 2$, and $\ell^2$ does not divide
$a_1^2-4a_2+8q$, or $\ell=2$, and $4$ does not divide
$a_1+a_2+1-2q$, then $Y=\Hp(2)$;
\item[(5)] if $f_A(t)\equiv(t-\alpha)^4\bmod\ell$, then
$A[\ell]\cong A(f,Y)$, where $\dim Y=4$.
\end{description}

Suppose that $f_A$ is not separable, then we have the following
three cases:
\begin{description}
\item[(6)] $f_A=P^2$, where $P$ is separable. Then
\begin{description}
\item[(a)] if $P$ is separable modulo $\ell$, then
$A[\ell]\cong A(P,0)\oplus A(P,0)$;
\item[(b)] if $P(t)\equiv (t-\alpha)^2\bmod\ell$, then
$A[\ell]\cong A(P,Y_1)\oplus A(P,Y_2)$, where $\dim Y_1=\dim Y_2=2$.
\end{description}

\item[(7)] $f_A(t) = (t\pm\sqrt{q})^2(t^2 - b t + q)$, where
$P_1(t)=t^2 - b t + q$ is separable. Let $P_2(t) =
(t\pm\sqrt{q})^2$.

\begin{description}
\item[(a)] If $P_1\not\equiv P_2\bmod\ell$, and $P_1$ is separable modulo $\ell$, then $A[\ell]\cong A(P_1,0)\oplus A(P_2,0)$.
\item[(b)]If $P_1\not\equiv P_2\bmod\ell$, and
$P_1(t)\equiv (t-\alpha)^2\bmod\ell$, then $A[\ell]\cong
A(P_1,Y)\oplus A(P_2,0)$, where $\dim Y=2$.
\item[(c)]If $P_1\equiv P_2\bmod\ell$, then
\begin{description}
\item[(i)]either $A[\ell]\cong
A(P_1(t)(t\pm\sqrt{q}),Y)\oplus A(t\pm\sqrt{q},0)$, where $\dim Y=3$; or
\item[(ii)]
if $\ell^2$ divides $P_1(\mp\sqrt{q})$, then $A[\ell]\cong
A(P_1,Y)\oplus A((t\pm\sqrt{q})^2,Y)$, where $Y=\Hp(2)$.
\end{description}
\end{description}

\item[(8)] If $f_A(t)=(t\pm\sqrt{q})^4$, then $A[\ell]\cong A(f_A,0)$.
\end{description}
Conversely, for any group scheme $G$ described above there exists
an abelian variety $B$ in the isogeny class of $A$ such that
$B[\ell]\cong G$.
\end{thm}
\begin{proof}
Assume that $f_A$ is separable. Note that if
$f_A(t)\equiv(t-\alpha)^3(t-\beta)\bmod\ell$, then
$\alpha\equiv\beta\bmod\ell$. Thus by
Corollaries~\ref{main_av_tors1} and~\ref{main_av_tors2} the cases
$(1)-(3)$ and $(5)$ follow. In the case $(4)$ we have
$A[\ell]\cong A(f_A, Y)$, where $\dim Y=2$, and $Y$ is the Young polygon of the zero matrix if and only if
$R=\ZZ_\ell[t]/f_A(t)\ZZ_\ell[t]$ is not a DVR. Indeed, if $R$ is
regular, then $T_\ell(A)$ is free, and hence the action of $t$ on $T_\ell(A)/\ell T_\ell(A)$ is non-trivial. 
If $R$ is not regular, then the integral closure $\OO$ of $R$ is an example
of an $R$-module such that $\OO/\ell\OO\cong A(f_A,0)(\kk)$. By the
Dedekind lemma~\cite[5.55]{ZP}, $R$ is regular if and only if
$(f_A-h^2)/\ell$ is prime to $h$ modulo $\ell$. An easy
computation shows that the two polynomials are coprime if and only
if $\ell^2$ does not divide $a_1^2-4a_2+8q$ for $\ell\neq 2$, and
$4$ does not divide $a_1+a_2+1-2q$ for $\ell=2$.

Assume now that  $f_A$ is not separable. It follows from the
classification of Weil polynomials (see~\cite{MN}), that only the
cases $(6)-(8)$ are possible. The case $(6)$ follows from the
Proposition~\ref{np_deg2}, and the case $(8)$ is obvious since
Frobenius acts as multiplication by $\mp\sqrt{q}$. By
Corollary~\ref{main_av_tors1}, the conditions of $(7a)$, $(7b)$ and
$(7c(i))$ are necessary. Let us prove that they are sufficient.
We have to construct Tate module $T$ with the prescribed Frobenius action. 
Then by Lemma~\ref{lem_on_Tate_module}, there exists an abelian variety $B$ in the isogeny class of $A$ such
that $T\cong T_\ell(B)$.

We give a construction for the case $(7b)$. Put
$T= T_1\oplus T_2$, where $T_i$ is a torsion-free module of rank $1$ over
$R_i=\ZZ_\ell[t]/P_i\ZZ_\ell[t]$. The module $T_1$ is uniquely
determined, and $T_2$ can be constructed using
Theorem~\ref{main2}. The case $(7a)$ is similar.

In the case $(7c(i))$ we construct the Tate module as the sum
$T= T_1\oplus T_2$, where $T_1$ is a module over
$R_1=\ZZ_\ell[t]/(t\pm\sqrt{q})\ZZ_\ell[t]$, and $T_2$ is a module
over $R_2=\ZZ_\ell[t]/P_1(t)(t\pm\sqrt{q})\ZZ_\ell[t]$. By
Theorem~\ref{main2}, for any $3\times 3$ nilpotent matrix $N$ such
that $\Np(P_1(t\mp\sqrt{q})t)$ lies on or above $\Hp(N)$ there exists
an $R_2$-module $T_2$ such that $t$ acts on $T_2/\ell T_2$ with
the matrix $N\mp\sqrt{q}I_3$. Then $T=R_1\oplus T_2$ is the
desired Tate module.

Suppose now that we have a module $T$ from the case $(7c(ii))$.
Let $P(t)=tP_1(t\mp\sqrt{q})$. By Proposition~\ref{prop_mf2}, there
exists a matrix factorization $(X,Y)$ such that $\det
X=f_A(t\mp\sqrt{q})$ and $YX=P(t)$. Moreover, $X\equiv
\diag(t^2,t^2)\mod\ell.$ Define $T'$ by the presentation:
\begin{equation}\label{eq_mf_2}
\ZZ_\ell[t]^2 \xrightarrow{Y} \ZZ_\ell[t]^2 \to T' \to 0,
\end{equation}

Note that $\det Y=P(t)^2/f_A(t\mp\sqrt{q})=P_1(t\mp\sqrt{q})$, and
$Y\equiv \diag(t,t)\mod\ell.$ Thus $T'$ is a module over
$R'=\ZZ_\ell[t]/P_1\ZZ_\ell[t]$. By Theorem~\ref{main1}, such a
module exists iff $\Np(P_1(t\mp\sqrt{q}))$ lies on or above the
Young polygon $\Hp(1,1)$. It follows that if $T$ exists then $\ell^2$
divides $P_1(\mp\sqrt{q})$. On the other hand if $\ell^2$ divides
$P_1(\mp\sqrt{q})$, then we can construct a module $T'$ over $R'$
such that $F$ acts on $T'/\ell T'$ with the matrix
$\mp\sqrt{q}I_2$. By Proposition~\ref{prop_mf2}, there exists a
matrix factorization $(Y,X)$ such that $\det Y=P_1(t\mp\sqrt{q})$
and $XY=P(t)$. Then the matrix factorization $(X,Y)$ corresponds
to a desired module $T$. 
\end{proof}

\section{Kummer surfaces}
Suppose $p\neq 2$. Let $A$ be an abelian surface, and let $\tau:
A\to A$ be the involution $a\mapsto -a$. Let $p_A:A\to A/\tau$ be
the quotient map. The variety $A/\tau$ is singular, and
$p_A(A[2])$ is the singular locus. Let $\sigma:S\to A/\tau$ be the blow
up of $A/\tau$ in $p_A(A[2])$. Then $S$ is smooth. It is called a {\it Kummer
surface}. In this section we compute zeta functions of Kummer
surfaces in terms of the zeta functions of the covering abelian
surfaces.

Let $X$ be a variety over a finite field $\fq$, and let $N_d$ be
the number of points of degree $1$ on $X\otimes\FF_{q^d}$. The
zeta function of $X$ is the formal power series
$$Z_X(t)=\exp(\sum_{d=1}^\infty \frac{N_dt^d}{d}).$$
In fact, $Z_X(t)$ is rational. For an
abelian variety $A$ we have the following formula:
\begin{equation}\label{zeta}
    Z_X(t)=\prod_{i=0}^{2g} P_i(t)^{(-1)^{i+1}},
\end{equation}
where $P_i(t)=\det(1-t\Fr|\bigwedge^i V_\ell(A))$. Note that if
$f_A(t)=\prod (t-\omega_j)$, then
$$P_i(t)=\prod_{j_1<\dots <j_i} (1-t\omega_{j_1}\dots\omega_{j_i}).$$
In particular, $P_1(t)=t^{2g}f_A(\frac1t)$, and $Z_A(t)=Z_B(t)$ if
and only if $A$ and $B$ are isogenous.

First we prove a general formula for the zeta function of a Kummer
surface $S$.

\begin{thm}
\label{kummer1} Let
$$Z_A(t)=\prod_{i=0}^{4} P_i(A,t)^{(-1)^{i+1}}$$
be the zeta function of an abelian surface $A$. Then
\begin{equation}\label{zeta_kummer}
    Z_S(t)=(1-t)^{-1}P(t)^{-1}(1-q^2t)^{-1},
\end{equation}
where
\begin{equation}\label{poly_kummer}
P(t)=P_2(A,t)\prod_{a\in A[2]}(1-(qt)^{\deg a}).
\end{equation}
In particular
\begin{equation}\label{number_of_points}
    |S(k)|=\frac{f_A(1)+f_A(-1)}{2}+q|A[2](k)|
\end{equation}
\end{thm}
\begin{proof}
Since $S$ is the blow up of $X$, we have $$Z_X(t)=Z_S(t)\prod_{a\in
A[2]}(1-(qt)^{\deg a}).$$ Let us prove that
$$|X(\FF_{q^r})|=\frac{f_r(1)+f_r(-1)}{2},$$
where $f_r=\det(t-\Fr^r)$ is the Weil polynomial of
$A_r=A\otimes\FF_{q^r}$. Put $$A(r)=A_r[2](\FF_{q^r}).$$ 
By~\cite[IV.19.4]{Mum}, $f_A(n)=\deg(n-\Fr)$ for $n\in\ZZ$, were $\deg$ means
the degree of an isogeny. There are two types of
possible fibers of the map $p_A$ over a nonsingular
$\FF_{q^r}$-point of $X$.
\begin{enumerate}
    \item The fiber is a union of two points of degree $1$. There are
    $\frac{f_r(1)-A(r)}{2}$ such fibers.
    \item The fiber is a point of degree $2$. There are
    $\frac{f_r(-1)-A(r)}{2}$ such fibers.
\end{enumerate}
This gives the desired equality.

Let $f_r(t)=t^4+a_1(r)t^3+a_2(r)t^2+a_1(r)q^rt+q^{2r}$, then
$a_2(r)=\tr(\Fr^r|H^2(\Bar A,\QQ_\ell))$, and
\begin{gather*}
Z_X(t)=\exp(\sum_{r=1}^\infty \frac{(f_r(1)+f_r(-1))t^r}{2r})=\\
\exp(\sum_{r=1}^\infty \frac{t^r}{r})
\exp(\sum_{r=1}^\infty \frac{a_2(r) t^r}{r})
\exp(\sum_{r=1}^\infty \frac{q^{2r} t^r}{r})=\\
(1-t)^{-1}P_2(A,t)^{-1}(1-q^2t)^{-1}
\end{gather*}
The last equality follows from lemma~C.4.1 of~\cite{Ha}.
\end{proof}

Now we classify the zeta functions of $A[2]$ in terms of the Weil
polynomial $f_A$. Let $b_r$ be the number of points of degree $r$
on $A[2]$. Then $P(t)=P_2(A,t)\prod_r(1-(qt)^r)^{b_r}$. We compute
the numbers $b_r$ using Theorem~\ref{surf_tors}.

Suppose first that $f_A$ is separable, and assume that
$f_A(t)\equiv(t+1)^4\bmod 2$. Note that the slopes of
$\Np(f_A(t+1))$ may be greater than $1$. This may create many
unnecessary cases in the table below. However, we can use the
polynomial $f(t)=f_A(t+\lambda)$ instead of $f_A(t+1)$, where
$\lambda\equiv 1\bmod\ell$, satisfying the property that slopes of
$\Np(f(t))$ are less then or equal to $1$. Equivalently, we take
$\Np(f_A(t+1))$ and change all its slopes that are greater than
$1$ to $1$. This operation simplifies the notation, and clearly,
it does not change the final answer, since all the slopes of Young
polygons are not greater than $1$.

\begin{center}
\begin{longtable}{|c|c|}
\caption{}\label{kummer_table1}\\
\hline
Slopes of $\Np(f(t))$  & $b_i$ \\
\hline \hline
$(1/4)$ & $b_1=2,b_2=1,b_4=3$\\
\hline
$(1/3,1)$ & $b_1=2,b_2=1,b_4=3$\\
    &$b_1=4,b_2=2,b_4=2$\\
\hline
$(1/2,1/2)$ & $b_1=2,b_2=1,b_4=3$\\
    &$b_1=4,b_2=2,b_4=2$\\
    &$b_1=4,b_2=6$\\
\hline
$(2/3,1)$,$(1/2,1,1)$ & $b_1=2,b_2=1,b_4=3$\\
or $(3/4)$ &$b_1=4,b_2=2,b_4=2$\\
    &$b_1=4,b_2=6$\\
    &$b_1=8,b_2=4$\\
\hline
$(1,1,1,1)$ & $b_1=2,b_2=1,b_4=3$\\
    &$b_1=4,b_2=2,b_4=2$\\
    &$b_1=4,b_2=6$\\
    &$b_1=8,b_2=4$\\
    &$b_1=16$\\
\hline
\end{longtable}
\end{center}

If $f_A(t)\not\equiv(t+1)^4\bmod 2$, then

\begin{center}
\begin{longtable}{|c|c|}
\caption{}\label{kummer_table2}\\
\hline
$f_A(t)\bmod 2$& \\
\hline
\hline
$t^4+t^3+t^2+t+1$       &$b_1=1, b_5=3$\\
\hline
$t^4+t^3+t+1$           &$b_1=2,b_2=1,b_3=2,b_6=1$\\
and $4$ does not divide $f_A(1)$ &\\
\hline
$t^4+t^3+t+1$           &$b_1=2,b_2=1,b_3=2,b_6=1$\\
and $4$ divides $f_A(1)$    &$b_1=4,b_3=4$\\
\hline
$t^4+t^2+1$             &$b_1=1,b_3=1,b_6=2$ \\
and $4$ does not divide $a_1+a_2+1-2q$ & \\
\hline
$t^4+t^2+1$             &$b_1=1,b_3=5$ \\
and $4$ divides $a_1+a_2+1-2q$ &$b_1=1,b_3=1,b_6=2$ \\
\hline
\end{longtable}
\end{center}

If $f_A$ is not separable, we have three cases of
theorem~\ref{surf_tors}. Let $f_A(t)=P_A(t)^2$ then

\begin{center}
\begin{longtable}{|c|c|}
\caption{}\label{kummer_table3}\\
\hline
$P_A(t)\bmod 2$& \\
\hline \hline
$t^2+t+1$               &$b_1=1, b_3=5$\\
\hline
$t^2+1$ and $4$ does not divide $P_A(1)$ &$b_1=4,b_2=6$\\
\hline
$t^2+1$             &$b_1=4,b_2=6$\\
and $4$ divides $P_A(1)$    &$b_1=8,b_2=4$\\
            &$b_1=16$\\
\hline
\end{longtable}
\end{center}

If $f_A(t)=(t\pm\sqrt{q})f(t)$, then
\begin{center}
\begin{longtable}{|c|c|}
\caption{}\label{kummer_table4}\\
\hline
$f(t)\bmod 2$& \\
\hline \hline
$t^2+t+1$               &$b_1=4, b_3=4$\\
\hline
$t^2+1$             &$b_1=8,b_2=4$\\
and $4$ does not divide $f(1)$ &$b_1=4,b_2=2,b_3=2$\\
\hline
$t^2+1$                &$b_1=16$\\
and $4$ divides $f(1)$  &$b_1=8,b_2=4$\\
           &$b_1=4,b_2=6$\\
         &$b_1=4,b_2=2,b_3=2$\\
\hline
\end{longtable}
\end{center}

Finally, if $f_A(t)=(t\pm\sqrt{q})^4$, we have $b_1=16$.

\end{document}